\documentclass{amsart}
\usepackage[ocgcolorlinks,linktoc=all]{hyperref}
\usepackage{lipsum}

\hypersetup{citecolor=blue,linkcolor=red}
\newtheorem{theorem}{Theorem}[section]

\newtheorem{lemma}[theorem]{Lemma}

\theoremstyle{definition}

\theoremstyle{remark}
\newtheorem{remark}[theorem]{Remark}

\numberwithin{equation}{section}

\begin{document}

\title[second mixed projection and projection centroid conjectures]
 {the second mixed projection problem and the projection centroid conjectures}
\author[M.N. Ivaki]{Mohammad N. Ivaki}
\address{Institut f\"{u}r Diskrete Mathematik und Geometrie, Technische Universit\"{a}t Wien, Wiedner Hauptstr. 8--10, 1040 Wien, Austria}
\email{mohammad.ivaki@tuwien.ac.at}

\dedicatory{}
\subjclass[2010]{}
\begin{abstract}
We provide partial answers to the open problems 4.5, 4.6 of \cite{Gardner} and 12.9 of \cite{Lut1} regarding the classification of fixed points of the second mixed projection operator and iterates of the projection and centroid operators.
\end{abstract}

\maketitle
\section{Introduction}
The setting of this paper is $n$-dimensional Euclidean space $\mathbb{R}^n.$
A compact convex subset of $\mathbb{R}^{n}$ with non-empty interior is called a \emph{convex body}. The set of convex bodies in $\mathbb{R}^{n}$ is denoted by $\mathrm{K}^n$. Write $\mathrm{K}^n_e$ for the set of origin-symmetric convex bodies. Also, write $B^n$ and $S^{n-1}$ for the unit ball and the unit sphere of $\mathbb{R}^n$. Moreover, $\omega_{k}$ denotes the volume of $B^k.$

The support function of $K\in \mathrm{K}^n$, $h_K:S^{n-1}\to\mathbb{R}$, is defined by $$h_{K}(u)=\max\limits_{x\in K}x\cdot u.$$

Assume $K\in \mathrm{K}^n$, $n\geq 2.$ The $i$th projection body $\Pi_i K$ of $K$ is the origin-symmetric convex body whose support function, for $u\in S^{n-1},$ is given by
\[h_{\Pi_i K}(u)=\frac{1}{2}\int\limits_{S^{n-1}}|u\cdot x|dS_i(K,x),\]
where $S_i(K,\cdot)$ is the mixed area measure of $i$ copies of $K$ and $n-1-i$ copies of $B^n;$ see \cite[Section A3]{Gardner}.
Note that $\Pi_{n-1}$ coincides with the usual projection operator $\Pi$. We refer the reader to \cite{Lut}, especially Proposition 2, regarding the importance of classification of solutions to $\Pi_i^2K=cK+\vec{v}$, where $c$ is a positive constant and $\vec{v}$ is a vector. Let us remark that $\Pi_i B^n=\omega_{n-1}B^n$ and $\Pi_i^2B^n=\omega_{n-1}^nB^n.$ \cite[Problems 4.6]{Gardner} and \cite[Problems 12.7]{Lut1} ask which convex bodies $K$ are such that $\Pi_i^2K$ is homothetic to $K.$ The case $i=n-1$ has received partial answers; see \cite{Ivaki,SZ,W1}. Schneider \cite{Sch} deals with the case  $i=1$ and proves origin-centered balls are the only solutions to $\Pi_1^2K=cK.$ Grinberg and Zhang \cite{GZ} provide an alternative path to this result. Motivated by the work of Fish, Nazarov, Ryabogin and Zvavitch \cite{FNRZ} where the idea of considering the iteration problems locally was first considered,
here we prove local uniqueness theorems for fixed points of the second mixed projection operators for $1<i<n-1$:
\begin{theorem}\label{thm1}
Suppose $n\ge 3$ and $1< i< n-1.$ There exists $\varepsilon>0$ with the following property. If a convex body $K$ satisfies $\Pi_i^2K=cK+\vec{v}$ for some $c>0$ and $\vec{v}\in \mathbb{R}^n$, and $\|h_{\lambda K+\vec{a}}-1\|_{C^{2}}\leq \varepsilon$ for some $\lambda>0$ and $\vec{a}\in \mathbb{R}^n$, then $K$ is a ball.
\end{theorem}
A set $K$ in $\mathbb{R}^n$ is called star-shaped if it is non-empty and if $[0,x]\subset K$ for every $x\in K.$ For a compact star-shaped set $K,$ the radial function $\rho_K$ is defined by
\[\rho_K(x)=\max\{\lambda\ge0; \lambda x\in K\},\quad x\in\mathbb{R}^n-\{0\}.\]
A compact star-shaped set with a positive continuous radial function is called a star body.

The polar body, $K^{\ast}$, of a convex body $K$ with the origin in its interior is the convex body defined by
\[K^{\ast}=\{x\in \mathbb{R}^n; x\cdot y\leq 1~\text{for all}~y\in K\}.\]
It follows from the definition that $\rho_{K^{\ast}}=\frac{1}{h_K}$ on $S^{n-1}.$

The centroid body of a star body $K$ is an origin-symmetric convex body whose support function, for $u\in S^{n-1},$ is given by
\[h_{\Gamma K}(u)=\int\limits_{S^{n-1}}|u\cdot x|\rho_K^{n+1}(x)dx.\]
By a result of Petty, the centroid body of a convex body is always of class $C^2_+$ and $\Gamma\phi K=\phi \Gamma K$ for all $\phi\in \operatorname{Sl}_n;$ see \cite{Petty} and \cite[Theorem 9.1.3]{Gardner}.
For $K\in \mathrm{K}^n_e,$ $\Gamma K$ is the locus of centroids of halves of $K$ formed by slicing $K$ by hyperplanes through the origin.

The curvature image of $K\in \mathrm{K}^n_e $, $\Lambda K$, is the origin-symmetric convex body whose positive, continuous curvature function (for the definition of the curvature function see \cite[page 545]{Schneider}) $f_{\Lambda K}$ is given by
\[f_{\Lambda K}=\frac{1}{h_{K}^{n+1}}.\]
Moreover, $\Lambda \phi K=\phi\Lambda K$ for $\phi\in\operatorname{Sl}_n;$ see \cite[Lemma 7.12]{lutn}.

Two conjectures of Lutwak stated in \cite[Problem 12.9]{Lut1} are as follows: 1) If the convex body $K$ is such that $K$ and $\Gamma (\Pi K)^{\ast}$ are dilates, must $K$ be an ellipsoid? From now on, for simplicity, we set $(\Pi_i K)^{\ast}=\Pi_i^{\ast}K$. 2) If the star body $K$ is such that $K$ and $(\Pi \Gamma K)^{\ast}$ are dilates, does it follow that $K$ is an ellipsoid? By Petty's regularity theorem for centroid bodies, if $\Gamma (\Pi K)^{\ast}$ and $K$ are dilates then $K$ is origin-symmetric and of class $C^2_+.$ Also, by a result of Martinez-Maure \cite{Yves}, the projection body of a convex body of class $C^2_+$ is $C^2_+.$ Thus, if $K$ is a star body such that $(\Pi \Gamma K)^{\ast}$ and $K$ are dilates, then $K$ must be an origin-symmetric convex body of class $C^2_+.$ Furthermore, if $K\in \mathrm{K}^2_e$, then $\Pi K=2K^{\frac{\pi}{2}}$ (rotation of $K$ counter-clockwise through $90^{\circ}$). Therefore, $\Gamma \Pi^{\ast} K=\Gamma \frac{1}{2}(K^{\frac{\pi}{2}})^{\ast}=\frac{1}{4}\Pi\Lambda K^{\frac{\pi}{2}}=\frac{1}{4}\Pi(\Lambda K)^{\frac{\pi}{2}}=\frac{1}{2}\Lambda K$.
Consequently, if $\Gamma \Pi^{\ast} K=cK$ for some positive constant $c$, then $\Lambda K=2cK$. By a result of Petty \cite[Lemma 8.1]{Petty1}, $K$ is an origin-centered ellipse. Similarly, if $K\in \mathrm{K}^2_e$ and $(\Pi \Gamma K)^{\ast}$ and $K$ are dilates, then $K$ is an origin-centered ellipse. In conclusion, the answer to both questions in $\mathbb{R}^2$ is positive. For higher dimensions, we prove the following results:
\begin{theorem}\label{thm2}
Suppose $n\ge 3$. There exists $\varepsilon>0$ with the following properties.
 \begin{enumerate}
   \item Let $K$ be a convex body $K$ such that $(\Gamma \Pi^{\ast})^2K=cK+\vec{v}$ for some $c>0$ and $\vec{v}\in \mathbb{R}^n$, and $\|h_{\phi K+\vec{a}}-1\|_{C^{2}}\leq \varepsilon$ for some $\phi\in \operatorname{Gl}_n$ and $\vec{a}\in \mathbb{R}^n$, then $K$ is an ellipsoid.
   \item Suppose $1\leq i<n-1.$ If a convex body $K$ satisfies $(\Gamma \Pi_i^{\ast})^2K=cK+\vec{v}$ for some $c>0$ and $\vec{v}\in \mathbb{R}^n$, and $\|h_{\lambda K+\vec{a}}-1\|_{C^{2}}\leq \varepsilon$ for some $\lambda>0$ and $\vec{a}\in \mathbb{R}^n$, then $K$ is a ball.
\item If a star body $K$ satisfies $(\Pi\Gamma K)^{\ast}=cK$ for some $c>0$ and $\|\rho_{\phi K}-1\|\leq \varepsilon$ for some $\phi\in \operatorname{Gl}_n$, then $K$ is an origin-centered ellipsoid.
 \end{enumerate}
\end{theorem}
\section*{Acknowledgment}
The work of the author was supported by Austrian Science Fund (FWF) Project M1716-N25 and the European Research Council (ERC) Project 306445. We would like to thank the referee and Emanuel Milman  for comments that led to improvements of this article.
\section{Preliminaries}
A convex body is said to be of class $C_+^2$ if its boundary hypersurface is two times continuously differentiable, in the sense of differential geometry, and has everywhere positive Gauss-Kronecker curvature.

Let $S_{k}$ be the group of all the permutations of the set $\{1,\cdots,k\}.$
The mixed discriminant of functions $f_i\in C^{2}(S^{n-1}), 1\leq i\leq n-1,$ is a multi-linear operator defined as
\begin{align}\label{def 1}
\mathcal{Q}(f_1,\cdots,f_{n-1}):=\frac{1}{(n-1)!}\sum_{\delta, \tau\in S_{n-1}}(-1)^{\operatorname{sgn}(\delta)+\operatorname{sgn}(\tau)}\prod_{i=1}^{n-1}(A[f_i])_{\delta(i)\tau(i)},
\end{align}
where in a local orthonormal frame of $S^{n-1}$ the entries of the matrix $A[f_k]$ are given by $(A[f_k])_{ij}=\nabla_i\nabla_jf_k+\delta_{ij}f_k$ and $\nabla$ is the covariant derivative on $S^{n-1}.$  From the above definition it follows that the operator $\mathcal{Q}$ is independent of the order of its arguments; see \cite[Lemma 2-12 ]{An} for other important properties of $\mathcal{Q}$.

We define the (ordered) mixed volume of $f_i\in C^{2}(S^{n-1}), 1\leq i\leq n,$ by
\[V(f_1,\cdots,f_n):=\frac{1}{n}\int\limits_{S^{n-1}}f_1\mathcal{Q}(f_2,\cdots,f_{n})dx.\]
In general we may have $V(f_1,\cdots,f_{n})\neq V(f_{\delta(1)},\cdots,f_{\delta(n)})$ for $\delta\in S_n,$ but if $A[f_k]$ are all positive definite then the equality holds; see \cite[Lemma 2-12]{An}.

The mixed projection of $f_i\in C^{2}(S^{n-1}), 1\leq i\leq n-1,$ is defined as
\begin{align*}\label{def 2}
\Pi(f_1,\cdots,f_{n-1})(u)&=\frac{1}{2}\int\limits_{S^{n-1}}|u\cdot x|\mathcal{Q}(f_1,\cdots,f_{n-1})(x)dx.
\end{align*}
\begin{remark}
Let $\{h_{K_i}\}_{1\leq i\leq n}$ be the support functions of convex bodies $\{K_i\}$ of class $C^2_+.$ Then
$V(h_{K_1},\cdots,h_{K_n})$ agrees with the usual definition of the mixed volume of $K_1,\cdots,K_n$ and also
$\Pi(h_{K_1},\cdots,h_{K_{n-1}})=h_{\Pi (K_1,\cdots,K_{n-1})}$.
\end{remark}
For convenience we will put
\begin{align*}
\mathcal{Q}(f,\cdots,f)&=\mathcal{Q}(f),\quad
\mathcal{Q}(\underbrace{f,\cdots,f}_{i~\text{times}},g,\cdots,g)=\mathcal{Q}_i(f,g),\\
\mathcal{Q}_i(f,1)&=\mathcal{Q}_i(f),\quad
\mathcal{Q}(\underbrace{f,\cdots,f}_{i-1~\text{times}},g,\underbrace{1,\cdots,1}_{n-1-i~\text{times}})=q_i(f,g),\\
V(f,\cdots,f)&=V(f),\quad
V(\underbrace{f,\cdots,f}_{i~\text{times}},g,\cdots,g)=V_i(f,g),\quad V_i(f,1)=V_i(f),\\
\Pi(f,\cdots,f)&=\Pi f,\quad
\Pi(\underbrace{f,\cdots,f}_{i~\text{times}},g,\cdots,g)=\Pi_i(f,g),\quad
\Pi_i(f,1)=\Pi_if.
\end{align*}
\subsection{Spherical harmonics}
Write $L^2(S^{n-1})$ for the Hilbert space of square-integrable real functions on $S^{n-1}$ equipped with scalar product
\[(f,g):=\int\limits_{S^{n-1}}fgdx.\]
Write $\|\cdot\|_2$ for the induced norm by this scalar product.

Spherical harmonics of degree $k$ are eigenfunctions of the spherical Laplace operator $\Delta$ with the eigenvalue $k(k+n-2).$ In fact, if $Y_k$ is such function then
\[\Delta Y_k=-k(k+n-2)Y_k.\]
The set $\mathcal{S}^k$ of spherical harmonics of degree $k$ is a vector subspace of $C(S^{n-1}).$ Moreover,
$\operatorname{dim}\mathcal{S}^k=N(n,k)=\frac{2k+n-2}{k+n-2}{k+n-2 \choose k}
.$ In each space $\mathcal{S}^k,$ choose an orthonormal basis
$\{Y_{k,1},\cdots,Y_{k,N(n,k)}\}$. For any $f\in L^2(S^{n-1})$ we write
\[\pi_kf:=\sum_{l=1}^{N(n,k)}\left(f,Y_{k,l}\right)Y_{k,l},\quad\pi_0f=\frac{1}{n\omega_n}\int\limits_{S^{n-1}}fdx.\]
The condensed harmonic expansion of $f$ is given by
\[f\sim \sum_{k=0}^{\infty}\pi_kf;\]
it converges to $f$ in the $L^2(S^{n-1})$-norm. In addition, for $f,g\in L^2(S^{n-1})$ we have
\[\sum_{k=0}^{\infty}\sum_{l=1}^{N(n,k)}\left(f,Y_{k,l}\right)\left(g,Y_{k,l}\right)=\left(f,g\right).\]
Note that $f\in L^2(S^{n-1})$ if and only if its condensed harmonic expansion satisfies
\[\sum_{k=0}^{\infty}\|\pi_kf\|_2^2<\infty.\]
For an excellent source on spherical harmonics see \cite{Groemer}.
\subsection{Radon transform and cosine transform}
Suppose that $f$ is a Borel function on $S^{n-1}.$ The spherical Radon transform (also known as the Funk Transform; see, for example, \cite{Helgason}) and cosine transform of $f$ are defined as follows
\begin{align*}
\mathcal{R}f(u)=\frac{1}{(n-1)\omega_{n-1}}\int\limits_{S^{n-1}\cap u^{\perp}}f(x)dx,\quad
\mathcal{C}f(u)=\int\limits_{S^{n-1}}|u\cdot x|f(x)dx.
\end{align*}
The transformations $\mathcal{R}$ and $\mathcal{C}$ are self-adjoint, in the sense that if $f$ and $g$ are bounded Borel functions on $S^{n-1},$ then
\[\int\limits_{S^{n-1}}f(x)\mathcal{R}g(x)dx=\int\limits_{S^{n-1}}g(x)\mathcal{R}f(x)dx,~\int\limits_{S^{n-1}}f(x)\mathcal{C}g(x)dx=
\int\limits_{S^{n-1}}g(x)\mathcal{C}f(x)dx.\]
Radon transform and cosine transform of a spherical harmonic of degree $k$ are given by
$$\mathcal{R}Y_k=v_{k,n}Y_k,\quad v_{k,n}=(-1)^{\frac{k}{2}}\cdot\left\{
                                    \begin{array}{ll}
                                      \frac{1\cdot3\cdots(k-1)}{(n-1)(n+1)\cdots(n+k-3)}, & k\geq 4~\hbox{even;} \\
                                      1, & k=0; \\
                                      \frac{1}{n-1}, & k=2; \\
                                      0, & k~ \hbox{odd}
                                    \end{array}
                                  \right.
$$
and
$$\mathcal{C}Y_k=w_{k,n}Y_k,\quad w_{k,n}=(-1)^{\frac{k-2}{2}}\omega_{n-1}\cdot\left\{
                                    \begin{array}{ll}
                                      2\frac{1\cdot3\cdots(k-3)}{(n+1)(n+3)\cdots(n+k-1)}, & k\geq 4~\hbox{even;} \\
                                      2, & k=0; \\
                                      \frac{2}{n+1}, & k=2; \\
                                      0, & k~ \hbox{odd},
                                    \end{array}
                                  \right.
$$
see \cite[Lemma 3.4.5, 3.4.7]{Groemer}. The following relation between Radon transform and $\Box:=\Delta+n-1$ is established in \cite[Proposition 2.1]{Goodey}:
\begin{align}\label{C and R}
\Box\mathcal{C}=2(n-1)\omega_{n-1}\mathcal{R}.
\end{align}
Let $H^s(S^{n-1}),$ $s\ge 0$, be the spaces of those functions for which the spherical harmonic expansion satisfies
$
\|f\|_{H^s}^2:=\sum_{k=0}^{\infty}(1+k^2)^s\|\pi_kf\|_2^2<\infty.
$
The following results about the smoothing property of $\mathcal{R}, ~\mathcal{C}$ are proved in \cite{STRICHARTZ}:
\begin{align}\label{STRICHARTZ R}
\|\mathcal{R}f\|_{H^{s+\frac{n-2}{2}}}\leq a_{s,n}\|f\|_{H^s},\quad
\|\mathcal{C}f\|_{H^{s+\frac{n+2}{2}}}\leq b_{s,n}\|f\|_{H^s}
\end{align}
for some positive constants depending on $s$ and $n.$
Let us put $$H_e^s(S^{n-1})=\{f\in H^s(S^{n-1}); f(x)=f(-x), \forall x\in S^{n-1}\}.$$ Strichartz proved that $H_e^s(S^{n-1})$ is precisely the even functions $f\in L^2(S^{n-1})$ with derivatives derivations up to order $s$ in $L^2(S^{n-1});$ see \cite[Pages 721-722]{STRICHARTZ}. Therefore, $H_e^s(S^{n-1})$, $s\ge0$, are Sobolev spaces.
\subsection{The directional derivative}
Let $B_1$ and $B_2$ be two Banach spaces, $U$ an open subset of $B_1.$ Suppose $P:U\subset B_1\to B_2$ is continuous. The directional derivative of $P$ at $f\in U$ in direction $g\in B_1$ is defined by
\[DP\{f,g\}=\lim_{t\to0}\frac{P(f+tg)-P(f)}{t}.\]
If the limit exists, $P$ is said to be differentiable at $f$ in direction $g.$ We say $P$ is $C^1$ in $U$, if the limit exists for all $f\in U$ and $g\in B_1$ and if $DP: (U\subset B_1)\times B_1\to B_2$ is continuous. The second derivative of $P$ is the derivative of the first derivative:
\[D^2P\{f,g_1,g_2\}=\lim_{t\to0}\frac{DP\{f+tg_2,g_1\}-DP\{f,g_1\}}{t}.\]
We say $P$ is $C^2$ if $D^2P$ exists and
$D^2P:(U\subset B_1)\times B_1\times B_1\to B_2$
is jointly continuous on the product. Similar definitions apply to the higher derivatives. The $k$th derivative $D^kP\{f,g_1,\cdots,g_k\}$ will be regarded as a map
\[D^kP:(U\subset B_1)\times \underbrace{B_1\times\cdots\times B_1}_{k~\text{times}}\to B_2.\]
We say $P$ is of class $C^k$, if $D^kP$ exists and is continuous. We say a map is $C^{\infty}$ if it is $C^k$ for all $k.$

If $P:U\subset B_1\to V\subset B_2$ is a map between open subsets of Banach spaces, we define its tangent map $TP:(U\subset B_1)\times B_1\to (V\subset B_2)\times B_2 $ by
\[TP(f,g)=(P(f), DP\{f,g\}).\]
Note that $TP$ is defined and continuous if and only if $P$ is $C^1$. Let $T^{k}P=T(T^{k-1}P)$, then $T^kP$ is defined and continuous if and only if $P$ is $C^k.$ If $P$ and $Q$ are $C^k$, then their composition is also $C^k$ and $T^k(P\circ Q)=T^kP\circ T^kQ;$ see \cite[Theorem 3.6.4]{Hamilton}.
\section{The mixed projection problem}
If $f\in C(S^{n-1})$, it follows from \cite[Theorem 1.1]{Yves} that $\mathcal{C}f \in C^2(S^{n-1}).$ In particular,
if $f\in C^2(S^{n-1})$, then $\Pi_i^kf \in C^2(S^{n-1}).$
\begin{lemma}\label{lemma reg}
The operator $\Pi_i^k:C^2(S^{n-1})\to C^2(S^{n-1})$ is $C^{\infty}$.
\end{lemma}
\begin{proof}
Since $\Pi_i=\frac{1}{2}\mathcal{C}\circ \mathcal{Q}_i$ and $\mathcal{C}: C(S^{n-1})\to C^2(S^{n-1})$ is linear, it suffices to show that the operator $\mathcal{Q}_i: C^2(S^{n-1})\to C(S^{n-1})$ is $C^m$ for all $m.$
Since $\mathcal{Q}$ is multi-linear, we have
\begin{align*}
\mathcal{Q}_i(f+tg_1)-\mathcal{Q}_i(f)
&=it\mathcal{Q}_i(\underbrace{f,\cdots,f}_{i-1~\text{times}},g_1,1,\cdots,1)+o(t^2).
\end{align*}
Therefore, \begin{align}\label{first derv}D\mathcal{Q}_i\{f,g_1\}=i\mathcal{Q}(\underbrace{f,\cdots,f}_{i-1~\text{times}},g_1,1\cdots,1)=iq_i(f,g_1).
\end{align}
Consequently,
$$D\mathcal{Q}_i\{f+tg_2,g_1\}-D\mathcal{Q}_i\{f,g_1\}=i(i-1)t\mathcal{Q}(\underbrace{f,\cdots,f}_{i-2~\text{times}},g_1,g_2,1,\cdots,1)+o(t^2).$$
By induction we obtain
$$D^m\mathcal{Q}_i\{f,g_1,\cdots,g_m\}=\left\{
  \begin{array}{ll}
   \frac{i!}{(i-m)!}\mathcal{Q}(\underbrace{f,\cdots,f}_{i-m~\text{times}},g_1,g_2,\cdots,g_m,1,\cdots,1) & \hbox{$m\leq i$;} \\
    0 & \hbox{$m>i$.}
  \end{array}
\right.
$$
This explicit expression shows that $D^m\mathcal{Q}_i$ is defined and continuous.
\end{proof}
Suppose $f\in C^2(S^{n-1})$ is the support function of a convex body of class $C^2_+.$
Let $\tilde{U}_{f}$ be a $C^2(S^{n-1})$-neighborhood of $0$ such that for every $g\in \tilde{U}_f$, $f+g$ is the support function of a convex body of class $C^2_+.$ The Corollary on page 13 of \cite{Yves} implies that $\Pi_i^k(f+g)$, for any $k,$ is the support function of a $C^2_+$ convex body.
\begin{lemma}\label{lem 1} Suppose $h\in C^2(S^{n-1})$ is the support function of a convex body of class $C^2_+.$ For $g\in C^2(S^{n-1})$ we have
\begin{align*}
\frac{d}{dt}\Big|_{t=0}\Pi^{k}_i(h+tg)=
\frac{i^k}{2^k}\mathcal{C}q_i(\Pi_i^{k-1}h,\mathcal{C}q_i(\Pi_i^{k-2}h,\mathcal{C}q_i(\cdots,\mathcal{C}q_i(h,g)\cdots)))
\end{align*}
and
$$\frac{d}{dt}\Big|_{t=0}\Pi^{k}_i(1+tg)=\frac{\left(\frac{V(\Pi^{k}1)}{\omega_n}\right)^{\frac{1}{n}}i^k}{2^k\omega_{n-1}^k(n-1)^k}(\mathcal{C}\Box)^{(k)} g,$$
where $(\mathcal{C}\Box)^{(k)} g=\underbrace{\mathcal{C}\Box\cdots\mathcal{C}\Box}_{k~\text{times}} g.$
Furthermore,
\begin{align*}
&\frac{d}{dt}\Big|_{t=0}V_{i+1}(\Pi^{k}_i(h+tg))\\
&=\frac{i^{k}(i+1)}{2^{k-1}n}\int\limits_{S^{n-1}}gq_i(h,\mathcal{C}q_i(\Pi_i h,
\mathcal{C}q_i(\cdots,\mathcal{C}q_i(\Pi^{k-1}_ih,\Pi^{k+1}_ih)\cdots)))dx
\end{align*}
and
\[\frac{d}{dt}\Big|_{t=0}V_{i+1}(\Pi_i^{k}(1+tg))=\frac{\left(\frac{V(\Pi^{k}1)}{\omega_n}\right)^{\frac{1}{n}}i^{k}(i+1)}{2^k\omega_{n-1}^k(n-1)^kn} \left((\Box\mathcal{C})^{(k)}(\Pi^{k}1)^{i}\right)\int\limits_{S^{n-1}}gdx.\]
\end{lemma}
\begin{proof}
Note that $T\Pi_i^{k}=\underbrace{T\Pi_i\circ\cdots\circ T\Pi_i}_{k~\text{times}}$ and by (\ref{first derv}) we have $$T\Pi_i(h,g)=(\Pi_ih,\frac{i}{2}\mathcal{C}q_i(h,g)).$$ Thus
\[T\Pi_i^{2}(h,g)=T\Pi_i(\Pi_ih,\frac{i}{2}\mathcal{C}q_i(h,g))=(\Pi_i^2h,\frac{i}{2}\mathcal{C}q_i(\Pi_ih,\frac{i}{2}\mathcal{C}q_i(h,g))).\]
The general claim follows by induction.

Note for a fixed $g$, there exists $\varepsilon>0$ small enough, such that for any $t\in (-\varepsilon,\varepsilon),$  $tg\in \tilde{U}_{h}$; therefore, $A[h+tg]$ is positive definite and it is the support function of a convex body of class $C^2_+$. To calculate $\frac{d}{dt}\Big|_{t=0}V_{i+1}(\Pi^{k}_i(h+tg))$, we may restrict our attention only to the range of $t\in (-\varepsilon,\varepsilon)$ (although, the definition of the derivative implicitly considers only small $t$ and so such care was not needed).
Recall that $$V_{i+1}(\Pi^{k}_i(h+tg))=\frac{1}{n}\int\limits_{S^{n-1}}\Pi_i^k(h+tg)\mathcal{Q}_i(\Pi_i^k(h+tg))dx.$$
Therefore,
\begin{align*}
\frac{d}{dt}\Big|_{t=0}V_{i+1}(\Pi^{k}_i(h+tg))=&\frac{1}{n}\int\limits_{S^{n-1}}\mathcal{Q}_i(\Pi_i^kh)\frac{d}{dt}\Big|_{t=0}\Pi_i^k(h+tg)
dx\\
&+\frac{1}{n}\int\limits_{S^{n-1}}\Pi_i^kh\frac{d}{dt}\Big|_{t=0}\mathcal{Q}_i(\Pi_i^k(h+tg))dx\\
=&\frac{1}{n}\int\limits_{S^{n-1}}\mathcal{Q}_i(\Pi_i^kh)\frac{d}{dt}\Big|_{t=0}\Pi_i^k(h+tg)
dx\\
&+\frac{i}{n}\int\limits_{S^{n-1}}\Pi_i^khq_i(\Pi_i^kh,\frac{d}{dt}\Big|_{t=0}\Pi_i^k(h+tg))dx.
\end{align*}
Using \cite[Lemma 2-12, items (3),(4)]{An}, we get
\begin{align*}
\frac{d}{dt}\Big|_{t=0}V_{i+1}(\Pi^{k}_i(h+tg))=
&\frac{i+1}{n}\int\limits_{S^{n-1}}\mathcal{Q}_i(\Pi_i^kh)\frac{d}{dt}\Big|_{t=0}\Pi_i^k(h+tg)dx.
\end{align*}
Since the operator $\mathcal{C}$ is self-adjoint and $\mathcal{C}\mathcal{Q}_i(\Pi^k_ih)=2\Pi_i^{k+1}h$, in view of \cite[Lemma 2-12, items (3),(4)]{An} the claim follows. For the special case of $h=1,$ we refer the reader to the proofs of \cite[Lemmas 3.2, 3.4]{Ivaki}.
\end{proof}
Fix $i,m\in \mathbb{N}.$ Suppose $f\in C^2(S^{n-1})$ is the support function of a convex body of class $C^2_+.$ Define a map by
\begin{align*}
\mathcal{X}_{i,f}^{m}&: \tilde{U}_{f}\subset C^2(S^{n-1})\to C^2(S^{n-1})\\
\mathcal{X}_{i,f}^{m}(g)(u):=&-\Pi_i^{m}(f+g)(u)+\left(\frac{V_{i+1}(\Pi_i^{m}(f+g))}{V_{i+1}(f+g)}\right)^{\frac{1}{1+i}}(f+g)(u)\\
&-u\cdot\int\limits_{S^{n-1}}\left(\frac{V_{i+1}(\Pi_i^{m}(f+g))}{V_{i+1}(f+g)}\right)^{\frac{1}{1+i}}(f+g)(x)xdx.
\end{align*}
By Lemmas \ref{lemma reg} and \ref{lem 1}, $\mathcal{X}_{i,f}^m$ is $C^{\infty}$. Lemma \ref{lem 1} yields an explicit expression for $D\mathcal{X}_{i,1}^{2m}\{0,\cdot\}:$
\begin{lemma}\label{lem: prev}
For any $g\in C^2(S^{n-1})$ we have
\begin{align*}
D\mathcal{X}_{i,1}^{2m}\{0,g\}(u)=&(\Pi^{2m}1)\left(g(u)-i^{2m}\mathcal{R}^{2m} g(u)+\frac{i^{2m}-1}{n\omega_n}\int\limits_{S^{n-1}} gdx\right)\\
&-(\Pi^{2m}1)u\cdot\int\limits_{S^{n-1}}g(x)xdx.
\end{align*}
Furthermore, if $1\leq i<n-1$, then $\operatorname{dim}\operatorname{Ker}D\mathcal{X}_{i,1}^{2m}\{0,\cdot\}=n+1.$
\end{lemma}
\begin{proof}Using Lemma \ref{lem 1}, we calculate
\begin{multline*}
D\mathcal{X}_{i,1}^{2m}\{0,g\}(u)=\Biggl(-\frac{\left(\frac{V(\Pi^{2m}1)}{\omega_n}\right)^{\frac{1}{n}}i^{2m}}{2^{2m}\omega_{n-1}^{2m}(n-1)^{2m}}(\mathcal{C}\Box)^{(2m)}g+\left(\frac{V_{i+1}(\Pi^{2m}1)}{\omega_n}\right)^{\frac{1}{i+1}}g\\
+\left(\frac{V_{i+1}(\Pi^{2m}1)}{\omega_n}\right)^{\frac{1}{i+1}-1}\frac{\left(\frac{V(\Pi^{2m}1)}{\omega_n}\right)^{\frac{1}{n}}i^{2m}}{2^{2m}\omega_{n-1}^{2m}(n-1)^{2m}} \left((\Box\mathcal{C})^{(2m)}(\Pi^{2m}1)^{i}\right)\frac{\int\limits_{S^{n-1}}gdx}{n\omega_n}\\
-\left(\frac{V_{i+1}(\Pi^{2m}1)}{\omega_n}\right)^{\frac{1}{i+1}}\frac{\int\limits_{S^{n-1}}gdx}{n\omega_n}\Biggl)(u)\\
-u\cdot \Biggl\{\int\limits_{S^{n-1}}\biggl( \left(\frac{V_{i+1}(\Pi^{2m}1)}{\omega_n}\right)^{\frac{1}{i+1}}g\\
+\left(\frac{V_{i+1}(\Pi^{2m}1)}{\omega_n}\right)^{\frac{1}{i+1}-1}\frac{\left(\frac{V(\Pi^{2m}1)}{\omega_n}\right)^{\frac{1}{n}}i^{2m}}{2^{2m}\omega_{n-1}^{2m}(n-1)^{2m}} \left((\Box\mathcal{C})^{(2m)}(\Pi^{2m}1)^{i}\right)\frac{\int\limits_{S^{n-1}}gdx}{n\omega_n}\\
-\left(\frac{V_{i+1}(\Pi^{2m}1)}{\omega_n}\right)^{\frac{1}{i+1}}\frac{\int\limits_{S^{n-1}}gdx}{n\omega_n}\biggl)(x)dx \Biggl\}.
\end{multline*}
On the other hand, $$(\Pi^{2m}1)^i=\left(\frac{V_{i+1}(\Pi^{2m}1)}{\omega_n}\right)^{-\frac{1}{i+1}+1}, \quad\left(\frac{V(\Pi^{2m}1)}{\omega_n}\right)^{\frac{1}{n}}=\left(\frac{V_{i+1}(\Pi^{2m}1)}{\omega_n}\right)^{\frac{1}{1+i}}.$$ Substituting these back into the above identity completes the proof.

To find the $\operatorname{dim}\operatorname{Ker}D\mathcal{X}_{i,1}^{2m}\{0,\cdot\},$ note that \begin{align}\label{steiner}u\cdot\int\limits_{S^{n-1}}g(x)xdx=\pi_1g(u);\end{align} see \cite[p. 50]{Schneider}. Therefore,
\begin{align}\label{kernel}
D\mathcal{X}_{i,1}^{2m}\{0,g\}=\left\{
  \begin{array}{ll}
    0, & g\in \mathcal{S}^0\oplus\mathcal{S}^1 ; \\
    (\Pi^{2m}1)(1-i^{2m}v_{k,n}^{2m})\pi_kg, & g\in\mathcal{S}^k\quad k\ge 2.
  \end{array}
\right.
\end{align}
Thus $\operatorname{dim}\operatorname{Ker}D\mathcal{X}_{i,1}^{2m}\{0,\cdot\}=n+1.$
\end{proof}
\begin{lemma}\label{lem 4}
Suppose $m\ge4$ and $1< i<n-1.$ Given $h\in C^{2}(S^{n-1}) $ with $\pi_kh=0$ for $k= 0, 1$, there exists a unique $g\in C^{2}(S^{n-1}) $ with $\pi_kg=0$ for $k= 0, 1$ such that
$$g-i^{2m}\mathcal{R}^{2m} g=h.$$
\end{lemma}
\begin{proof}
We develop $h$ into a series of spherical harmonics:
$h\sim \sum_{k\ge 2}^{\infty} \pi_{k}h.$
Since $L^2(S^{n-1})$ is a complete space and $\lim\limits_{k\to\infty }1-i^{2m}v_{k,n}^{2m}= 1$, the $L^2(S^{n-1})$-Cauchy sequence
\begin{align*}
\left\{f_l:=\sum_{k\ge 2}^{l} \frac{1}{1-i^{2m}v_{k,n}^{2m}}\pi_{k}h\right\}_l
\end{align*}
converges in the $L^2(S^{n-1})$-norm to a bounded $f\in L^2(S^{n-1})\cap \left(\mathcal{S}^0\oplus\mathcal{S}^1\right)^{\perp}$ with
$$\pi_{k}f=\frac{1}{1-i^{2m}v_{k,n}^{2m}}\pi_{k}h$$ for $k\ge 2.$ In view of (\ref{STRICHARTZ R}), $\mathcal{R}^{2m}f\in H_e^{m(n-2)}\subset H_e^{4(n-2)}\subset C^{2}(S^{n-1}).$
Define $$g:=h+i^{2m}\mathcal{R}^{2m} f.$$ Note that $g\in C^2(S^{n-1})\cap \left(\mathcal{S}^0\oplus\mathcal{S}^1\right)^{\perp}$ and for $k\ge 2:$
\[\pi_{k}g=\left(1+\frac{i^{2m}v_{k,n}^{2m}}{1-i^{2m}v_{k,n}^{2m}}\right)\pi_{k}h\Rightarrow\pi_{k}(g-i^{2m}\mathcal{R}^{2m}g)=\pi_{k}h.\]
Since $h$ and $g-i^{2m}\mathcal{R}^{2m}g$ are $C^2,$ we conclude that
\begin{align}\label{x}
g-i^{2m}\mathcal{R}^{2m}g=h
\end{align}
The uniqueness claim: Suppose $g_1,g_2\in C^{2}(S^{n-1})$ both solve (\ref{x}) with $\pi_kg_i=0$ for $k=0,1.$ Therefore, for $k\ge2,$
\[(1-i^{2m}v_{k,n}^{2m})\pi_kg_i=\pi_kh\Rightarrow \pi_kg_i=\frac{\pi_kh}{1-i^{2m}v_{k,n}^{2m}}\Rightarrow \pi_kg_1=\pi_kg_2\Rightarrow g_1=g_2. \]
\end{proof}
\begin{theorem}\label{main prop}
Suppose $m\ge 4$ and $1< i<n-1$ There exists $\varepsilon_{m}>0$, such that if $K$ satisfies
$\Pi^{2m}_iK=cK+\vec{v}$ for some $c>0$ and $\vec{v}\in \mathbb{R}^n$, and
$\|h_{\lambda K+\vec{a}}-1\|_{C^{2}}\leq \varepsilon_{m}$ for some $\lambda>0$ and $\vec{a}\in \mathbb{R}^n,$ then $K$ is a ball.
In particular, if $\Pi^{2}_iK=cK+\vec{v}$ for some $c>0$ and $\vec{v}\in \mathbb{R}^n$, and $\|h_{\lambda K+\vec{a}}-1\|_{C^{2}}\leq \varepsilon_{4}$ for some $\lambda>0$ and $\vec{a}\in \mathbb{R}^n,$ then $K$ is a ball.
\end{theorem}
\begin{proof}
Fix $1< i<n-1$ and $m\geq 4.$ In this proof, $B$ always denotes a ball.
Consider the map
\begin{align*}
\mathcal{N}:\tilde{U}_1\subset C^2(S^{n-1})\to C^2(S^{n-1})\quad f\mapsto\mathcal{X}_{i,1}^{2m}(f)+(\pi_0+\pi_1)(f),\quad \mathcal{N}(0)=0.
\end{align*}
\begin{itemize}
  \item $\mathcal{N}$ is $C^{\infty}$.
  \item $D\mathcal{N}\{0,\cdot\}:C^2(S^{n-1})\to C^2(S^{n-1})$ is an invertible linear map:
  \begin{enumerate}
    \item By (\ref{kernel}), we have
    \begin{align*}D\mathcal{N}\{0,f\}&=\mathcal{X}_{i,1}^{2m}\{0,f\}+(\pi_0+\pi_1)f\\
    &=(\Pi^{2m}1)\sum_{k\geq 2}(1-i^{2m}v_{k,n}^{2m})\pi_kf+\pi_0f+\pi_1f.
        \end{align*}
        Thus, $D\mathcal{N}\{0,f\}=0$ implies that $f=0.$
    \item Given $h\in C^{2}(S^{n-1})$, by Lemma \ref{lem 4} there exists a unique $g\in C^{2}(S^{n-1})$ such that $$(\Pi^{2m}1)(g-i^{2m}\mathcal{R}^{2m}g)=h-(\pi_0+\pi_1)h$$ and $\pi_kg=0$ for $k=0,1$. Define $l=g+(\pi_0+\pi_1)h$. Then $l\in C^2(S^{n-1})$ and $D\mathcal{N}\{0,l\}=h$.
  \end{enumerate}

\end{itemize}
By the inverse function theorem (see, \cite[Theorem 5.2.3, Corollary 5.3.4]{Hamilton}), we can find neighborhoods $U,~W$ of $0$ in $C^{2}(S^{n-1})$, such that $\mathcal{N}: U\subset \tilde{U}_1\to W$
is a smooth diffeomorphism. Put $M:=\mathcal{N}^{-1}(W\cap \operatorname{Ker}D\mathcal{X}_{i,1}^{2m}\{0,\cdot\})$. Observe that
$$f\in U~\text{and}~\mathcal{X}_{i,1}^{2m}(f)=0\Rightarrow\mathcal{N}^{-1}((\pi_0+\pi_1) (f))=f\in M.$$
On the other hand, if $h_B-1\in U,$ then $\mathcal{X}_{i,1}^{2m}(h_B-1)=0$ and clearly $$W':=\{(\pi_0+\pi_1)(h_B-1);~ h_B-1\in U\}=\{h_B-1;~ h_B-1\in U\}$$ forms an open subset of $W\cap \operatorname{Ker}D\mathcal{X}_{i,1}^{2m}\{0,\cdot\}$ about the origin in $\mathbb{R}^{n+1}$.

Define the open set
\[W_1:=\{f\in W;(\pi_0+\pi_1)f\in W'\}.\]
Thus $W'=W_1\cap \operatorname{Ker}D\mathcal{X}_{i,1}^{2m}\{0,\cdot\}$ and $\mathcal{N}^{-1}(W')=M\cap \mathcal{N}^{-1}(W_1).$ Since $\mathcal{N}^{-1}(W_1)$ is an open neighborhood of $0$ in $C^2(S^{n-1})$ and $\mathcal{N}^{-1}(W')=\{h_B-1;~h_B-1\in U\},$ we conclude that in a $C^{2}$-neighborhood of $1$, the only solutions of $\mathcal{X}_{i,1}^{2m}(\cdot-1)=0$ are balls.

So far we have shown that there exists $\varepsilon_{m}>0$, such that if $K$ satisfies $\mathcal{X}_{i,1}^{2m}(h_K-1)=0$ and
 $\|h_{K}-1\|_{C^{2}}\leq \varepsilon_{m}$, then $K$ is a ball. To see that the first claim of the theorem holds, note that if $\Pi^{2m}_iK=cK+\vec{v}$, then
$\mathcal{X}_{i,1}^{2m}(h_{\lambda K+\vec{a}}-1)=0$; therefore, $K$ is a ball.

To prove the second statement of the theorem, note that $\Pi^2_iK=cK+\vec{v}$ yields $\mathcal{X}_{i,1}^{8}(h_{\lambda K+\vec{a}}-1)=0;$ therefore, $K$ is a ball.
\end{proof}
\section{The projection centroid conjectures}
For a convex body $K$ of class $C^2_+$, define
$$\Theta_i(h_K):=h_{\Gamma\Pi_i^{\ast}K}=\mathcal{C}((\Pi_ih_K)^{-(n+1)})=2^{n+1}\mathcal{C}\frac{1}{(\mathcal{C} \mathcal{Q}_i(h_K))^{n+1}}.$$ Recall that $\mathcal{C}\mathcal{Q}_i(h_K)=2\Pi_i h_K$ is the support function of an origin-symmetric convex body of class $C^2_+$ and consequently $2/\mathcal{C}\mathcal{Q}_i(h_K)$ is the radial function (restricted to $S^{n-1}$) of the corresponding polar body (which is also an origin-symmetric convex body of class $C^2_+$). Petty's regularity theorem for centroid bodies ensures $\Theta_i(h_K)$ is the support function of a convex body of class $C^2_+$. By induction, for any $m\in \mathbb{N},$ $$\Theta^m_i(h_K):=\underbrace{\Theta_i\circ\cdots\circ \Theta_i}_{m~\text{times}} (h_K)$$ is the support function of a convex body of class $C^2_+$.

Define the map
\begin{align*}
\mathcal{Y}_{i,1}^{m}&: \tilde{U}_1\subset C^2(S^{n-1})\to C^2(S^{n-1})\\
\mathcal{Y}_{i,1}^{m}(f)(u)= &\left(-\Theta_i^{m}(1+f)+\left(\frac{V(\Theta_i^m (1+f))}{V(1+f)}\right)^{\frac{1}{n}}(1+f)\right)(u)\\
&-u\cdot\int\limits_{S^{n-1}}\left(\frac{V(\Theta_i^m (1+f))}{V(1+f)}\right)^{\frac{1}{n}}(1+f)(x)xdx.
\end{align*}
Since $\Pi_i: C^2(S^{n-1})\to C^2(S^{n-1})$ is $C^{\infty},$ by the chain rule the map $\mathcal{Y}_{i,1}^m$ is $C^{\infty}$. 
Furthermore, we  calculate
\begin{align*}
T\Theta_i(h_K,g)&=2^{n+1}T\mathcal{C}\circ T \frac{1}{x^{n+1}}\circ T\mathcal{C}\circ T\mathcal{Q}_i(h_K,g)\\
&=2^{n+1}T\mathcal{C}\circ T \frac{1}{x^{n+1}}\circ T\mathcal{C}(\mathcal{Q}_i(h_K),iq_{i}(h_K,g))\\
&=2^{n+1}T\mathcal{C}\circ T \frac{1}{x^{n+1}}(\mathcal{C}\mathcal{Q}_i(h_K),i\mathcal{C}q_{i}(h_K,g))\\
&=2^{n+1}T\mathcal{C}(\frac{1}{\left(\mathcal{C}\mathcal{Q}_i(h_K)\right)^{n+1}},-i(n+1)\frac{\mathcal{C}q_{i}(h_K,g)}{\left(\mathcal{C}\mathcal{Q}_i(h_K)\right)^{n+2}})\\
&=(\Theta_i (h_K),-i(n+1)2^{n+1}\mathcal{C}\frac{\mathcal{C}q_{i}(h_K,g)}{\left(\mathcal{C}\mathcal{Q}_i(h_K)\right)^{n+2}})
\end{align*}
and
\begin{align}\label{cal}
T\Theta_i^2(h_K,g)&=T\Theta_i\circ T\Theta_i(h_K,g)\nonumber\\
&=T\Theta_i(\Theta_i (h_K),-i(n+1)2^{n+1}\mathcal{C}\frac{\mathcal{C}q_{i}(h_K,g)}{\left(\mathcal{C}\mathcal{Q}_i(h_K)\right)^{n+2}})\nonumber\\
&=(\Theta_i^2 (h_K),-i(n+1)2^{n+1}\mathcal{C}\frac{\mathcal{C}q_{i}(\Theta_i(h_K),-i(n+1)2^{n+1}\mathcal{C}\frac{\mathcal{C}q_{i}(h_K,g)}
{\left(\mathcal{C}\mathcal{Q}_i(h_K)\right)^{n+2}})}{\left(\mathcal{C}\mathcal{Q}_i(\Theta_i(h_K))\right)^{n+2}}).
\end{align}
\begin{lemma}\label{c1}
For any $g\in C^2(S^{n-1})$ we have
\begin{align*}
D\mathcal{Y}_{i,1}^{2}\{0,g\}(u)=&\Theta_i^2(1)\left(g(u)-\frac{i^2(n+1)^2}{4\omega_{n-1}^2}\mathcal{C}^2\mathcal{R}^{2} g(u)+\frac{i^2(n+1)^2-1}{n\omega_n}\int\limits_{S^{n-1}} gdx\right)\\
&-\Theta_i^2(1)u\cdot\int\limits_{S^{n-1}}g(x)xdx.
\end{align*}
Also, we have $$\operatorname{dim}\operatorname{Ker}D\mathcal{Y}_{i,1}^{2}\{0,\cdot\}=\left\{
               \begin{array}{ll}
                 \frac{n(n+3)}{2}, & i=n-1; \\
                 n+1, & 1\leq i<n-1.
               \end{array}
             \right.$$
\end{lemma}
\begin{proof}
Note that $\Theta_i(h_{\lambda K})=\frac{1}{\lambda^{i(n+1)}}\Theta_i(h_{K}).$ Hence we get
$$\Theta_i^2(1)=\Theta_i(\Theta_i(1))=\frac{1}{(\Theta_i(1))^{i(n+1)-1}}.$$
In view of $\Theta_i(1)=2/\omega_{n-1}^n$ (e.q., $\mathcal{C}1=2\omega_{n-1}$) and the identity (\ref{cal}) for $h_K\equiv1$ we obtain
\begin{align}\label{a}
\frac{1}{\Theta_i^2(1)}\frac{d}{dt}\Big|_{t=0}\Theta_i^{2}(1+tg)=&\frac{1}{\Theta_i^2(1)}\frac{i^22^{2n+4}(n+1)^2(\Theta_i(1))^{i-1}\omega_{n-1}^2}{(\Theta_i(1))^{i(n+2)}(2\omega_{n-1})^{2n+4}}\mathcal{C}^2\mathcal{R}^2g\nonumber\\
=&\frac{i^2(n+1)^2}{(\Theta_i(1))^{2}(\omega_{n-1})^{2n+2}}\mathcal{C}^2\mathcal{R}^2g\nonumber\\
=&\frac{i^2(n+1)^2}{4\omega_{n-1}^2}\mathcal{C}^2\mathcal{R}^2g.
\end{align}
Using this last expression, we  calculate
\begin{align*}
\frac{1}{\Theta_i^2(1)}\frac{d}{dt}\Big|_{t=0}V(\Theta_i^{2}(1+tg))&=\int\limits_{S^{n-1}}\frac{i^2(n+1)^2}{4\omega_{n-1}^2}\mathcal{C}^2\mathcal{R}^2g\mathcal{Q}(\Theta_i^2(1))dx\\
&=(\Theta_i^2(1))^{n-1}i^2(n+1)^2\int\limits_{S^{n-1}}gdx.
\end{align*}
Thus
\begin{align}\label{b}
\frac{1}{\Theta_i^2(1)}\frac{d}{dt}\Big|_{t=0}\left(\frac{V(\Theta_i^2 (1+tg))}{V(1+tg)}\right)^{\frac{1}{n}}=\frac{i^2(n+1)^2-1}{n\omega_n}\int\limits_{S^{n-1}} gdx.
\end{align}
Putting (\ref{a}), (\ref{b}) together yields the explicit expression for $D\mathcal{Y}_{i,1}^{2}\{0,g\}.$
In view of (\ref{steiner}), it is easy to check that $D\mathcal{Y}_{i,1}^{2}\{0,g\}=0$ for all $g\in \mathcal{S}^0\oplus\mathcal{S}^1.$ Also, from $\mathcal{C}g=\frac{2\omega_{n-1}}{n+1}g$ and $\mathcal{R}g=-\frac{1}{n-1}g$ for all $g\in \mathcal{S}^2,$ it follows that
$\mathcal{S}^0\oplus\mathcal{S}^1\oplus\mathcal{S}^2\subset \operatorname{Ker}D\mathcal{Y}_{n-1,1}^{2}\{0,\cdot\}.$
To complete the proof, note that if $k\ge3$ and $i=n-1$ or $k\ge2$ and $i<n-1$, then $1-\frac{i^2(n+1)^2}{4\omega_{n-1}^2}v_{k,n}^2w_{k,n}^2\neq 0$.
\end{proof}
\begin{lemma}\label{c2}The following statements hold.
\begin{enumerate}
  \item Given $h\in C^{2}(S^{n-1}) $ with $\pi_kh=0$ for $k= 0,1,2$, there exists a unique $g\in C^{2}(S^{n-1})$ with $\pi_kg=0$ for $k= 0,1,2$ such that
$$g-\frac{(n-1)^2(n+1)^2}{4\omega_{n-1}^2}\mathcal{C}^2\mathcal{R}^{2} g=h.$$
  \item Suppose $1\leq i<n-1$. Given $h\in C^{2}(S^{n-1}) $ with $\pi_kh=0$ for $k=0,1$, there exists a unique $g\in C^{2}(S^{n-1}) $ with $\pi_kg=0$ for $k= 0,1$ such that
$$g-\frac{i^2(n+1)^2}{4\omega_{n-1}^2}\mathcal{C}^2\mathcal{R}^{2} g=h.$$
\end{enumerate}
\end{lemma}
\begin{proof}
We only give the proof of the first claim.
Develop $h$ into a series of spherical harmonics:
$h\sim \sum_{k\ge 3}^{\infty} \pi_{k}h.$
The $L^2(S^{n-1})$-Cauchy sequence
\begin{align*}
\left\{f_l:=\sum_{k\ge 3}^{l} \frac{1}{1-\frac{(n-1)^2(n+1)^2}{4\omega_{n-1}^2}v_{k,n}^{2}w_{k,n}^2}\pi_{k}h\right\}_l
\end{align*}
converges in the $L^2(S^{n-1})$-norm to a bounded $f\in L^2(S^{n-1})\cap \left(\mathcal{S}^0\oplus\mathcal{S}^1\oplus\mathcal{S}^2\right)^{\perp}$ with
$$\pi_{k}f=\frac{1}{1-\frac{(n-1)^2(n+1)^2}{4\omega_{n-1}^2}v_{k,n}^{2}w_{k,n}^2}\pi_{k}h$$ for $k\ge 3.$ In view of (\ref{STRICHARTZ R}), $\mathcal{C}^2\mathcal{R}^{2}f\in H_e^{2n}\subset C^{2}(S^{n-1}).$
Set $$g:=h+\frac{(n-1)^2(n+1)^2}{4\omega_{n-1}^2}\mathcal{C}^2\mathcal{R}^{2} f.$$ We have $g\in C^2(S^{n-1})\cap \left(\mathcal{S}^0\oplus\mathcal{S}^1\oplus\mathcal{S}^2\right)^{\perp}$ and for $k\ge 3$
\[\pi_{k}g=\left(1+\frac{\frac{(n-1)^2(n+1)^2}{4\omega_{n-1}^2}v_{k,n}^{2}w_{k,n}^2}{1-\frac{(n-1)^2(n+1)^2}{4\omega_{n-1}^2}v_{k,n}^{2}w_{k,n}^2}\right)\pi_{k}h.\]
Therefore,
\[\pi_{k}(g-\frac{(n-1)^2(n+1)^2}{4\omega_{n-1}^2}\mathcal{C}^2\mathcal{R}^{2}g)=\pi_{k}h.\]
Since $h,g-\frac{(n-1)^2(n+1)^2}{4\omega_{n-1}^2}\mathcal{C}^2\mathcal{R}^{2}g\in C^2(S^{n-1}),$ we obtain
$$g-\frac{(n-1)^2(n+1)^2}{4\omega_{n-1}^2}\mathcal{C}^2\mathcal{R}^{2}g=h.$$
The uniqueness claim is trivial.
\end{proof}
Given Lemmas \ref{c1} and \ref{c2}, proofs of the first and second statements in Theorem \ref{thm2} are straightforward; if $1\leq i<n-1$, we precisely follow the proof of Theorem \ref{thm1} and for the case $i=n-1$, we refer the reader to the proof of \cite[Theorem 4.2]{Ivaki}.

All constants $c_i$ that follow are positive. We proceed to the proof of the third statement. Note that for any $L\in \mathrm{K}^n_e:$
\begin{align*}
\Gamma L=2\Pi \Lambda L^{\ast}, \quad \Gamma \psi K=\psi\Gamma K~ \text{for any} ~\psi\in\operatorname{Sl}_n.
\end{align*}
We may assume without loss of generality that $\phi\in\operatorname{Sl}_n$. We have
\[(\Pi\Gamma\phi K)^{\ast}=c_1\phi K\Rightarrow \Pi\phi\Gamma  K=\frac{1}{c_1}(\phi K)^{\ast}.\]
Therefore,
\[\Pi\Lambda\Pi\phi\Gamma K=c_2\Pi\Lambda (\phi K)^{\ast}=\frac{c_2}{2}\Gamma\phi K\Rightarrow \Gamma\Pi^{\ast}(\phi\Gamma K)=c_3\phi\Gamma K.\]
On the other hand, if $\delta$ is small enough, then from $\|\rho_{\phi K}-1\|\leq \delta$ it follows that
$\|\rho_{\phi K}^{n+1}-1\|\leq 2\delta.$ Thus by \cite[Ineq. (2.3)]{Ivaki},
\[\|h_{\frac{1}{2\omega_{n-1}}\phi \Gamma K}-1\|_{C^2}=\frac{1}{2\omega_{n-1}}\|\mathcal{C}\rho_{\phi K}^{n+1}-\mathcal{C}1\|_{C^2}\leq c_n\delta.\]
In summary, we have shown that $\Gamma\Pi^{\ast}(\frac{1}{2\omega_{n-1}}\phi\Gamma K)=\frac{c_4}{2\omega_{n-1}}\phi\Gamma K$ and $h_{\frac{1}{2\omega_{n-1}}\phi\Gamma K}$ satisfies the assumption (1) in Theorem \ref{thm2} provided $\delta$ is small enough. Hence $\phi\Gamma K$ is an origin-centered ellipsoid. Since $\Pi\phi\Gamma K=\frac{1}{c_1}(\phi K)^{\ast},$ $K$ is an origin-centered ellipsoid.

\bibliographystyle{amsplain}

\end{document}